\documentclass[sn-mathphys,Numbered]{sn-jnl}% Math and Physical Sciences Reference Style
\usepackage{graphicx}%
\usepackage{multirow}%
\usepackage{amsmath,amssymb,amsfonts}%
\usepackage{amsthm}%
\usepackage{mathrsfs}%
\usepackage[title]{appendix}%
\usepackage{xcolor}%
\usepackage{textcomp}%
\usepackage{manyfoot}%
\usepackage{booktabs}%
\usepackage{algorithm}%
\usepackage{algorithmicx}%
\usepackage{algpseudocode}%
\usepackage{listings}%
\usepackage{latexsym}
\usepackage{color}
\usepackage{bm}
\usepackage{graphicx}
\usepackage{subfigure}
\usepackage{epstopdf}
\usepackage[misc]{ifsym}
\usepackage{caption}
\usepackage{xcolor}
\usepackage{xpatch}

\theoremstyle{thmstyleone}%
\newtheorem{theorem}{Theorem}%  meant for continuous numbers
\theoremstyle{thmstyletwo}%
\newtheorem{example}{Example}%
\newtheorem{remark}{Remark}%
\newtheorem{lemma}{Lemma}

\newtheorem{corollary}{Corollary}

\theoremstyle{thmstylethree}%
\newtheorem{definition}{Definition}%
\allowdisplaybreaks[4]

\begin{document}

\title[Clifford-valued linear canonical Stockwell transform]{Clifford-valued linear canonical Stockwell transform}


\author[1,2]{\fnm{Yi-Qiao} \sur{Xu}}

\author*[1,2]{\fnm{Bing-Zhao} \sur{Li}}\email{li\_bingzhao@bit.edu.cn}

\affil[1]{\orgdiv{School of Mathematics and Statistics}, \orgname{Beijing Institute of Technology},\orgaddress{\city{ Beijing}, \postcode{100081}, \country{China}}}

\affil[2]{\orgdiv{Beijing Key Laboratory on MCAACI}, \orgname{Beijing Institute of Technology},\orgaddress{\city{ Beijing}, \postcode{100081}, \country{China}}}


\abstract{We present a new Clifford-valued linear canonical Stockwell transform aimed at providing efficient and focused representation of Clifford-valued functions in high-dimensional time-frequency analysis. This transform improves upon the windowed Fourier and wavelet transforms by incorporating angular, scalable, and localized windows, allowing for greater directional flexibility in multi-scale signal analysis within the Clifford domain. Using operator theory, we explore the core properties of the proposed transform, such as the inner product relation, reconstruction formula, and interval theorem. Practical examples are included to confirm the validity of the derived results.}




\maketitle

\section{Introduction}
\label{sec1}

Wavelet transforms break down non-stationary signals into components based on the expansion and translation of a single generating function, the mother wavelet. Their introduction has revolutionized multi-scale analysis, attracting significant attention from both scientific and mathematical communities due to their broad applicability and solid mathematical foundation. However, despite their extensive use in fields like signal and image processing, differential equations, sampling theory, quantum mechanics, and medicine, wavelet transforms have two main drawbacks \cite{bib7, bib8, bib10}: first, while wavelet transforms are fundamentally time-scale transforms where inverse scales are interpreted as frequencies, the details they capture do not directly correlate with frequencies. Second, wavelet transforms completely lose phase information when the wavelet is shifted along the time axis. To address these issues, Stockwell \cite{bib23} proposed the Stockwell transform, which serves as a link between short-time Fourier transforms and wavelet transforms. The Stockwell transform of any function $f\in L^2(\mathbb{R})$ with respect to a window function $\psi$ is defined as follows:

\begin{equation}\label{eq1}
(\mathbb{S}_\psi f)(b,u)=\frac{|u|}{\sqrt{2\pi}}\int_\mathbb{R}e^{-it\cdot u}f(t) \overline{\psi\left(u(t-b)\right)} dt,\quad b\in\mathbb{R}, u\in\mathbb{R}\backslash\{0\}.
\end{equation}
Here, $b$ and $u$ represent the time of spectral localization and the Fourier frequency, respectively. By choosing $\psi(t)=e^{-t^2/2}$ as the analysing window function, the integral transform (\ref{eq1}) can be rewritten as:
\begin{equation}\label{eq2}
(\mathbb{S}_\psi f)(b,u)=\frac{|u|}{\sqrt{2\pi}}\int_{-\infty}^\infty e^{-itu}f(t) e^{-(t-b)^2u^2/2} dt,\quad b\in\mathbb{R}, u\in\mathbb{R}\backslash\{0\}.
\end{equation}

The window function $e^{-(t-b)^2u^2/2}$, as described in (\ref{eq2}), contracts at high frequencies and expands at low frequencies, improving time-frequency resolution and warranting further exploration. This function offers a non-temporal representation of signals. Over the past few decades, the Stockwell transform has achieved considerable success and gained a strong foundation in various fields, including geophysics, optics, quantum mechanics, acoustics, feature recognition, biomedical imaging, oceanography, and signal processing \cite{bib11, bib13, bib17, bib24}. To extend its application to high-dimensional signal processing, several researchers have modified the Stockwell transform. For example, Liu and Wong \cite{bib15} adapted it for two dimensions and examined various variants of the associated reconstruction formula. Meanwhile, Riba and Wong \cite{bib16, bib17} conducted an in-depth study of its multidimensional application. Additionally, Shah and Tantary \cite{bib18} introduced the anisotropic angular Stockwell transform, which incorporates local and scalable angular window functions in time-frequency analysis. Srivastava et al. \cite{bib22} proposed a fractional version of the Stockwell transform, leveraging its convolutional structure to improve signal analysis in the fractional frequency domain. Later, Shah and Tantary \cite{bib19} extended the standard Stockwell transform to the linear domain and applied it to quasi-periodic functions.

Clifford algebra offers a more powerful alternative to Grossmann's external algebra and Hamilton's quaternion algebra by integrating both geometric and algebraic aspects of Euclidean space into its structure \cite{bib12, bib21}. As a result, Clifford algebra has gained significant attention and made substantial inroads in high-dimensional signal and image processing. Unlike the multidimensional tensor method, which uses tensor products of one-dimensional phenomena, Clifford algebra simultaneously handles all dimensions, allowing for true multidimensional representations that improve performance. This multidimensional capability enables the development of advanced time-frequency analysis tools, offering a comprehensive approach to signal processing across various fields.

The linear canonical transform (LCT) features three free parameters, allowing for rotation and expansion in the time-frequency domain. This flexibility makes it particularly effective for analysing non-stationary functions \cite{bib26, bib44}. The Fourier transform (FT), fractional Fourier transform (FrFT), and Fresnel transforms are all special cases of the LCT \cite{bib45}. Shah et al. extended the LCT to Clifford-valued signals, introducing the linear canonical Clifford-valued transform \cite{bib37}. To effectively represent Clifford-valued signals using the Stockwell transform, Shah proposed the Clifford-valued Stockwell transform in the context of time-frequency analysis \cite{bib1}. Building on the Clifford-valued Stockwell and linear canonical transforms, this paper introduces a new Clifford-valued linear canonical Stockwell transform. This transform is both simple in structure and physically meaningful. It aims to improve the short-time Fourier and Clifford-valued wavelet transforms by using angular, scalable, and localized windows. The paper explores the core properties of the proposed transform, including dot product relations, reconstruction formulas, and interval theorems, employing operator theory and the Clifford-valued Fourier transform.

The structure of this paper is as follows: Section \ref{sec2} offers an introduction to Clifford algebras and relevant literature. In Section \ref{sec3}, we define the Clifford-valued linear canonical Stockwell transform and explore its key properties in detail. Section \ref{sec4} provides an example to illustrate and clarify our findings.

\section{Preliminaries}
\label{sec2}

This section reviews fundamental notations of Clifford algebra, including the definitions of the Clifford-valued Fourier transform and Clifford-valued Stockwell transform, along with some of their key properties.

\subsection{Clifford algebra}
\label{subsec1}

The Clifford algebra $C\ell_{(0,n)}:=C\ell_n$ is a non-commutative associative algebra generated by the orthonormal basis $n$-dimensional Euclidean space $\{e_1,e_2,\ldots,e_n\}$, with the following multiplication rule:
\begin{equation}\label{eq3}
e_ie_j+e_je_i=-2\delta_{ij},\quad i,j=1,2,\ldots,n.
\end{equation}
Here, $\delta_{ij}$ represents the Kronecker's delta function. The non-commutative product and the associativity axiom give rise to the $2^n$-dimensional Clifford geometric algebra $C\ell_n$, which can be expressed as:
\begin{equation}\label{eq4}
C\ell_n=\bigoplus_{k=0}^nC\ell_n^k.
\end{equation}
Here, $C\ell_n^k$ refers to the space of $k$-vectors, defined as:
\begin{equation}\nonumber
C\ell_n^k:=\mathrm{span}\Big\{e_{i_1},e_{i_2},\ldots,e_{i_k}; i_1\leq i_2\leq\cdots\leq i_k\Big\}.
\end{equation}

A general element of the Clifford algebra $C\ell_n$ is referred to as a multi-vector. Each multi-vector $M \in C\ell_n$ can be expressed as:
\begin{equation}\label{eq5}
\begin{aligned}
&M=\sum_AM_Ae_A=\langle M\rangle_0+\langle M\rangle_1+\cdots+\langle M\rangle_n,\\
&M_A\in\mathbb{R}, A\subset\big\{1,2,\ldots,n\big\}.
\end{aligned}
\end{equation}
Here, $e_A = e_{i_1}e_{i_2}\ldots e_{i_k}$ and $i_1 \leq i_2 \leq \cdots \leq i_k$. Additionally, $\langle\cdot\rangle_{k}$ denotes the grade $k$-component of $M$, while $\langle\cdot\rangle_0, \langle\cdot\rangle_1, \langle\cdot\rangle_2,\ldots $ represent the scalar, vector, bi-vector parts, and so on, respectively. The Clifford conjugate of $M \in C\ell_n$ is defined as:
\begin{equation}\label{eq6}
\overline{M}=\sum_{r=0}^n(-1)^{\frac{r(r-1)}2}\overline{\langle M\rangle_r}.
\end{equation}

%We now recall the fundamental notion of Clifford Fourier transformation on the space of functions $L^p(\mathbb{R}^n, C\ell_n), 1\leq p<\infty $, defined as
%\begin{equation}\nonumber
%L^p\big(\mathbb{R}^n, C\ell_n\big)=\left\{f:\mathbb{R}^n\to C\ell_n;\left\|f\right\|_p=\left(\int_{\mathbb{R}^n}\left|f(\mathbf{x})\right|^pd^n\mathbf{x}\right)^{1/p}<\infty\right\}.
%\end{equation}
%
%It is imperative to mention that any Clifford-valued function $f \in L^p(\mathbb{R}^n, C\ell_n)$ can be expressed as a combination of the real-valued functions $f_{A}$ and the basis elements $e_{A}$ as
%\begin{equation}\nonumber
%f(\mathbf{x})=\sum_{A}f_{A}(\mathbf{x})e_{A}.
%\end{equation}
%
%Due to the non-commutativity of Clifford-valued functions, different types of Clifford-valued Fourier transforms have been introduced in recent years [6]. However, we shall be interested in following definition due to Bahri and Hitzer [1].

\subsection{Related works}
\label{subsec2}

Next, we review the definition of the Clifford-valued Fourier transform (CFT).

\begin{definition}\label{df1}
The Clifford-valued Fourier transform of any function $f \in L^2(\mathbb{R}^n, C\ell_n), n=2,3(\text{mod 4})$ is denoted by $\mathcal{F}_{C\ell}$ and is defined as
\begin{equation}\label{eq7}
\mathcal{F}_{C\ell}\big[f(\mathbf{x})\big](\mathbf{w})=\frac1{(2\pi)^{n/2}} \int_{\mathbb{R}^n}f(\mathbf{x}) e^{-\mathbf{i}_n \mathbf{w}\cdot\mathbf{x}} d\mathbf{x},
\end{equation}
where $\mathrm{x},\mathrm{w}\in\mathbb{R}^n$ and $\mathbf{i}_n=e_1e_2\ldots e_n\in C\ell_n$.
\end{definition}
Definition \ref{df1} leads to the following observations:

(1). Notably, the Clifford exponential product $e^{-\mathbf{i}_n}\mathbf{w}\cdot\mathbf{x}$ in (\ref{eq7}) commutes with every element of $L^2(\mathbb{R}^n, C\ell_n)$ when $n=3(\mathrm{mod~}4)$, but is non-commutative when $n=2(\mathrm{mod~}4)$.

(2). For any $f,g\in L^2(\mathbb{R}^n, C\ell_n)$, the Plancherel formula reads
\begin{equation}\label{eq8}
\left\langle f,g\right\rangle_{L^2(\mathbb{R}^n,C\ell_n)}=\left\langle\mathcal{F}_{C\ell}\begin{bmatrix}f\end{bmatrix},\mathcal{F}_{C\ell}\begin{bmatrix}f\end{bmatrix}\right\rangle_{L^2(\mathbb{R}^n,C\ell_n)}.
\end{equation}

(3).The Clifford-valued function $f \in L^2(\mathbb{R}^n, C\ell_n)$ appearing in (\ref{eq7}) can be reconstructed via
\begin{equation}\label{eq9}
f(\mathbf{x})=\frac1{(2\pi)^{n/2}}\int_{\mathbb{R}^n}\mathcal{F}_{C\ell}\big[f(\mathbf{x})\big](\mathbf{w}) e^{\mathbf{i}_n \mathbf{w}\cdot\mathbf{x}} d\mathbf{w}.
\end{equation}

Next, we shall formally introduce a transform namely the Clifford-valued Stockwell transform in the context of higher-dimensional time-frequency analysis.

For $\mathbf{u}=(u_1,u_2,\ldots, u_n)\in\mathbb{R}^n$, $\mathbf{b}=(b_1,b_2,\ldots,b_n)\in\mathbb{R}^n$, and $\theta\in\mathrm{SO}(n)$, the special orthogonal group of $\mathbb{R}^n$, we define a family of Clifford-valued functions $\psi_{\mathbf{b},\mathbf{u}}^\theta(\mathbf{x})$ in $L^2(\mathbb{R}^n, C\ell_n)$ as:
\begin{equation}\label{eq10}
\psi_{\mathbf{b},\mathbf{u}}^\theta(\mathbf{x})=\left|\det A_\mathbf{u}\right| e^{i_n\cdot\mathbf{x}\cdot\mathbf{u}}\psi\Big(R_{\boldsymbol{-}\theta}A_\mathbf{u}(\mathbf{x}-\mathbf{b})\Big),
\end{equation}
where $A_{\mathbf{u}}=\begin{pmatrix}u_{11}&u_{12}&\ldots&u_{1n}\\u_{21}&u_{22}&\ldots&u_{2n}\\\vdots&\vdots&\ldots&\vdots\\u_{n1}&u_{n2}&\ldots&u_{nn}\end{pmatrix},u_{ij}\in\mathbb{R}, \left|\det A_{\mathbf{u}}\right|\neq0,$ and $R_{-\theta}=\begin{pmatrix}\cos\theta&\sin\theta&\ldots&0\\\sin\theta&\cos\theta&\ldots&0\\\vdots&\vdots&\ldots&\vdots\\0&0&\ldots&1\end{pmatrix}$.

For simplicity in computations, we choose the matrix $A_{\mathbf{u}}$ of the following form:
\begin{equation}\label{eq11}
A_{\mathbf{u}}=\begin{pmatrix}
u_{1}&0&\ldots&0\\
0&u_{2}&\ldots&0\\
\vdots&\vdots&\ldots&\vdots\\
0&0&\ldots&u_{n}\end{pmatrix},\quad u_1,u_2,\ldots, u_n\neq0.
\end{equation}

The matrix $A_{\mathbf{u}}$ given by (\ref{eq11}) satisfies the following properties:

(1).$A_{\mathbf{u}}^T=A_{\mathbf{u}};$

(2).$A_\mathbf{u}\mathbf{x}=A_\mathbf{u}^T\mathbf{x}=(x_1u_1,x_2,\ldots,x_nu_n), \mathbf{x}\in\mathbb{R}^n;$

(3).$\begin{aligned}\left|\det A_{\mathbf u}\right|=\prod_{i=1}^{n}u_{i}\quad\text{ and }\quad\left|\det A_{\mathbf u}^{-1}\right|=\left|\det A_{\mathbf u}\right|^{-1};\end{aligned}$

(4).$A_{\mathbf{u}}^{-1}=\mathrm{diag}\left(\frac{1}{u_{1}},\frac{1}{u_{2}},\dots,\frac{1}{u_{n}}\right);$

(5).$\begin{aligned}A_{\mathbf{u}}^{-1}\mathbf{x}=\left(\frac{x_1}{u_1},\frac{x_2}{u_2},\dots,\frac{x_n}{u_n}\right)^T,\quad\mathbf{x}\in\mathbb{R}^n.\end{aligned}$

Next, we introduce the definition of Clifford-valued Stockwell transform.

\begin{definition}\label{df2}
The Clifford-valued Stockwell transform (CST) of any function $f \in L^2(\mathbb{R}^n, C\ell_n)$ with respect to a window function $\psi\in L^1(\mathbb{R}^n,C\ell_n)\cap L^2(\mathbb{R}^n,C\ell_n)$, is defined by
\begin{equation}\label{eq12}
\left(\mathbb{S}_\psi f\right)(\mathbf{b},\mathbf{u},\theta)=\frac{\left|\det A_\mathbf{u}\right|}{(2\pi)^{n/2}} \int_{\mathbb{R}^n}f(\mathbf{x}) \overline{\psi\left(\mathbf{R}_{-\theta}\mathbf{A}_\mathbf{u}(\mathbf{x}-\mathbf{b})\right)} \mathbf{e}^{-\mathbf{i}_\mathbf{n}\mathbf{x}\cdot\mathbf{u}}d\mathbf{x},
\end{equation}
where $\int_{\mathbb{R}^n}\psi(\mathbf{x})d\mathbf{x}=1$.
\end{definition}

Furthermore, CST can be expressed using classical convolution as:
\begin{equation}\label{eq13}
\left(\mathbb{S}_\psi f\right)(\mathbf{b},\mathbf{u},\theta)=\frac{1}{(2\pi)^{n/2}}\Big[e^{i_n\cdot\mathbf{x}\cdot\mathbf{u}}f(\mathbf{x})\Big]*\Big[\left|\det A_\mathbf{u}\right|\tilde{\psi}\Big(R_{\boldsymbol{-}\theta}A_\mathbf{u}(\mathbf{x}-\mathbf{b})\Big)\Big](\mathbf{b}),
\end{equation}
where $\tilde{\psi}(\mathbf{x})=\overline{\psi}(-\mathbf{x})$.

\begin{definition}\label{df3}
Assume the parameter matrix $M=(A,B,C,D)$, then the Clifford linear canonical transform (CLCT) of the function $f \in L^2(\mathbb{R}^n, C\ell_n)$ is defined as:
\begin{equation}\label{eq14}
L_{M}\left(f\left(\mathbf{x}\right)\right)(\mathbf{u})=L_{f}^{M}(\mathbf{u})=\left\{
\begin{matrix}\int_{\mathbb{R}^n}f\left(\mathbf{x}\right)K_{M}\left(\mathbf{u},\mathbf{x}\right)d\mathbf{x}&B\neq0\\
D^{-\frac{n}{2} }\mathbf{e}^{-\mathbf{i}_\mathbf{n}{\frac{CD}{2}\mathbf{u}^{2}}}f\left(D\mathbf{u}\right)&B=0
\end{matrix}\right.,
\end{equation}
and the kernel function $K_{M}\left(\mathbf{u},\mathbf{x}\right)$ is given by:
\begin{equation}\label{eq15}
K_{M}\left(\mathbf{u},\mathbf{x}\right)=C_M\mathbf{e}^{\mathbf{i}_\mathbf{n}(\frac{A}{2B}\mathbf{x}^2-\frac{1}{B}\mathbf{x}\cdot\mathbf{u}+\frac{D}{2B}\mathbf{u}^2)},
\end{equation}
where $C_M=\frac{1}{\sqrt{(2\pi)^nB}}$.
\end{definition}

Convolution is an important basic concept in linear time-invariant systems. The output of a continuous-time linear time-invariant system is essentially the convolution of the input signal and the system function. Signals can be detected and processed through signal correlation operations [1,77]. Convolution has been widely used in signal processing, image processing, pattern recognition and other fields. The time domain convolution under Clifford-valued Fourier transform is expressed as:
\begin{equation}\label{eq16}
f(\mathbf{x})*g(\mathbf{x})=\int_{\mathbb{R}^n}f(\mathbf{\tau})g(\mathbf{x}-\mathbf{\tau})d\mathbf{\tau}.
\end{equation}

Then in the Clifford Fourier domain, the classical convolution theorem is expressed as:
\begin{equation}\label{eq17}
\mathcal{F}_{C\ell}[f(\mathbf{x})*g(\mathbf{x})](\mathbf{w})=\mathcal{F}_{C\ell}[f(\mathbf{x})](\mathbf{w})\mathcal{F}_{C\ell}[g(\mathbf{x})](\mathbf{w}).
\end{equation}

CLCT is a further extension of CFT. Here we introduce a simple linear canonical domain convolution theorem. The linear canonical domain convolution of functions $f$ and $g$ can be defined as:
\begin{equation}\label{eq18}
[f \Theta_M g](\mathbf{x})=\mathbf{e}^{-\mathbf{i}_\mathbf{n}\frac{A}{2B}\mathbf{x}^2}[(f(\mathbf{x})\mathbf{e}^{\mathbf{i}_\mathbf{n}\frac{A}{2B}\mathbf{x}^2})*g(\mathbf{x})],
\end{equation}
where $\Theta_M$ represents the convolution operator. The above convolution process is shown in Fig. \ref{fig1}. The linear canonical domain convolution theorem is expressed as:
\begin{equation}\label{eq19}
L_{M}\left[f \Theta_{M} g\right](\mathbf{u})=L_{M}[f](\mathbf{u})\mathcal{F}_{C\ell}[g](\frac{\mathbf{u}}{B}),
\end{equation}
where $L_{M}[f](\mathbf{u})$ and $\mathcal{F}_{C\ell}[g](\frac{\mathbf{u}}{B})$ represent the CLCT of $f(\mathbf{x})$ and the CFT of $g(\mathbf{x})$, respectively.

\begin{figure}
\centering
\includegraphics[width=0.9\textwidth]{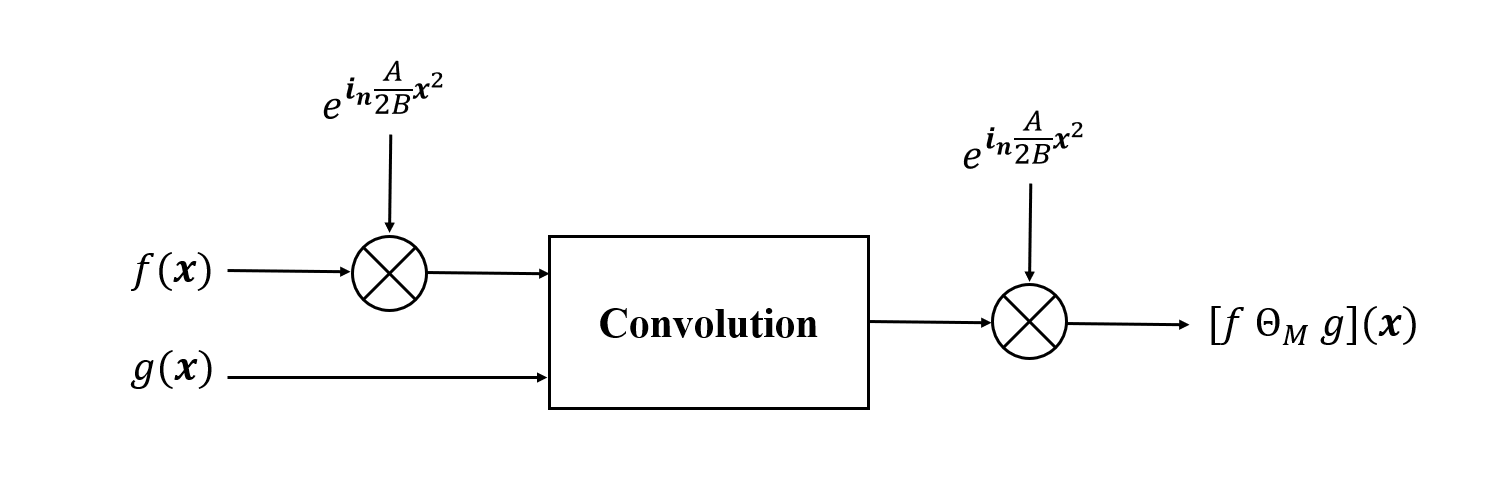}
\caption{The convolution process in the CLCT domain.}
\label{fig1}
\end{figure}

\section{Clifford-valued linear canonical Stockwell transform}
\label{sec3}
\subsection{Definition}
\label{subsec3}

\begin{definition}\label{df4}
According to the basic characteristics of CST and the convolution theorem of CLCT, the CLCST of function $f \in L^2(\mathbb{R}^n, C\ell_n)$ is defined as:
\begin{equation}\label{eq20}
\begin{aligned}
\left(CLC\mathbb{S}_{\psi}^{M} f\right)(\mathbf{b},\mathbf{u},\theta)
&=\frac{1}{(2\pi)^{n/2}}\Big(\Big[\mathbf{e}^{\mathbf{i}_\mathbf{n}\mathbf{x}\cdot\mathbf{u}}f(\mathbf{x})\Big] \Theta_M \Big[\left|\det A_\mathbf{u}\right|\overline{\psi{\left(\mathbf{R}_{-\theta}\mathbf{A}_\mathbf{u}(\mathbf{x}-\mathbf{b})\right)}}\Big]\Big)(\mathbf{b})\\
&=\frac{1}{(2\pi)^{n/2}}\mathbf{e}^{-\mathbf{i}_\mathbf{n}\frac{A}{2B}\mathbf{b}^2}\left\langle \mathbf{e}^{-\mathbf{i}_\mathbf{n}\mathbf{u}(\cdot)}f(\cdot)\mathbf{e}^{\mathbf{i}_\mathbf{n}\frac{A}{2B}(\cdot)^2},\psi(\cdot)\right\rangle\\
&=\frac{1}{(2\pi)^{n/2}}\int_{\mathbb{R}^n}f(\mathbf{x})\overline{\psi_{M,\mathbf{b},\mathbf{u}}^{\theta}}d\mathbf{x},\\
\end{aligned}
\end{equation}
where the kernel function $\psi_{M,\mathbf{b},\mathbf{u}}^{\theta}$ satisfies:
\begin{equation}\label{eq21}
\psi_{M,\mathbf{b},\mathbf{u}}^{\theta}(\mathbf{x})=\left|\det A_\mathbf{u}\right|\mathbf{e}^{\mathbf{i}_\mathbf{n}(\mathbf{x}\cdot\mathbf{u}+\frac{A}{2B}\mathbf{b}^2-\frac{A}{2B}\mathbf{x}^2)}\psi{\left(\mathbf{R}_{-\theta}\mathbf{A}_\mathbf{u}(\mathbf{x}-\mathbf{b})\right)}.
\end{equation}
\end{definition}

Therefore, formula (\ref{eq20}) can be written as:
\begin{equation}\label{eq22}
\left(CLC\mathbb{S}_{\psi}^{M} f\right)(\mathbf{b},\mathbf{u},\theta)=\frac{1}{(2\pi)^{n/2}}\left\langle f(\mathbf{x}) ,\psi_{M,\mathbf{b},\mathbf{u}}^{\theta}\right\rangle_{L^2(\mathbb{R}^n, C\ell_n)}.
\end{equation}

It can be seen that CLCST is the inner product of function $f(\mathbf{x})$ and window function $\psi_{M,\mathbf{b},\mathbf{u}}^{\theta}$. In particular, when the parameter satisfies $M=(0,1,-1,0)$, CLCST degenerates into traditional CST.

\begin{theorem}\label{thm1}
Let $F_M(\mathbf{u})$ and $L_M(\psi(\mathbf{x})\mathbf{e}^{\mathbf{i}_\mathbf{n}\mathbf{x}})$ represent the CLCT of $f(\mathbf{x})$ and $\psi(\mathbf{x})\mathbf{e}^{\mathbf{i}_\mathbf{n}\mathbf{x}}$, then we can get the form of CLCST in the linear canonical domain:
\begin{equation}\label{eq23}
\left(CLC\mathbb{S}_{\psi}^{M} f\right)(\mathbf{b},\mathbf{u},\theta)=\frac{1}{(2\pi)^{n/2}}\mathbf{e}^{-\mathbf{i}_\mathbf{n}\mathbf{u}\cdot\mathbf{b}}\int_{\mathbb{R}^n}F_M(\mathbf{u})\overline{L_M(\psi(\mathbf{x})\mathbf{e}^{\mathbf{i}_\mathbf{n}\mathbf{x}})}\quad \overline{K_M(\mathbf{b},\mathbf{u})}d\mathbf{u}.
\end{equation}
\end{theorem}

\begin{proof}
Perform CLCT on both sides of equation (\ref{eq21}) to obtain
\begin{align*}
&L_{M}(\psi_{M,\mathbf{b},\mathbf{u}}^{\theta}(\mathbf{x}))(\mathbf{u})\\
&=\int_{\mathbb{R}^n}\psi_{M,\mathbf{b},\mathbf{u}}^{\theta}(\mathbf{x})K_{M}(\mathbf{u},\mathbf{x})d\mathbf{x} \\
&=\int_{\mathbb{R}^n} \left|\det A_\mathbf{u}\right|\mathbf{e}^{\mathbf{i}_\mathbf{n}(\mathbf{x}\cdot\mathbf{u}+\frac{A}{2B}\mathbf{b}^2-\frac{A}{2B}\mathbf{x}^2)}\psi{\left(\mathbf{R}_{-\theta}\mathbf{A}_\mathbf{u}(\mathbf{x}-\mathbf{b})\right)}C_{M}\mathbf{e}^{\mathbf{i}_\mathbf{n}(\frac{A}{2B}\mathbf{x}^2-\frac{1}{B}\mathbf{u}\cdot\mathbf{x}+\frac{D}{2B}\mathbf{u}^2)}d\mathbf{x} \\
&=\left|\det A_\mathbf{u}\right|C_{M}\int_{\mathbb{R}^n} \psi{\left(\mathbf{R}_{-\theta}\mathbf{A}_\mathbf{u}(\mathbf{x}-\mathbf{b})\right)}\mathbf{e}^{\mathbf{i}_\mathbf{n}(\mathbf{x}\cdot\mathbf{u}-\frac{1}{B}\mathbf{u}\cdot\mathbf{x})}\mathbf{e}^{\mathbf{i}_\mathbf{n}(\frac{A}{2B}\mathbf{b}^2+\frac{D}{2B}\mathbf{u}^2)}d\mathbf{x} \\
&=C_{M}\int_{\mathbb{R}^n}\psi(\mathbf{z})\mathbf{e}^{\mathbf{i}_\mathbf{n}\mathbf{u}\cdot(\frac{\mathbf{z}}{\left|\det A_\mathbf{u}\right|}+\mathbf{b})}\mathbf{e}^{-\mathbf{i}_\mathbf{n}\mathbf{u}\cdot(\frac{\mathbf{z}}{\left|\det A_\mathbf{u}\right|}+\mathbf{b})\frac{1}{B}}\mathbf{e}^{\mathbf{i}_\mathbf{n}(\frac{A}{2B}\mathbf{b}^2+\frac{D}{2B}\mathbf{u}^2)}d\mathbf{z} \\
&=\mathbf{e}^{\mathbf{i}_\mathbf{n}\mathbf{u}\cdot\mathbf{b}}F_M(\psi(\mathbf{z})\mathbf{e}^{\mathbf{i}_\mathbf{n}\mathbf{z}})K_{M}(\mathbf{b},\mathbf{u})d\mathbf{u}.
\end{align*}

According to the Parseval property of CLCT, we can get
\begin{equation}\label{eq24}
\begin{aligned}
\left(CLC\mathbb{S}_{\psi}^{M} f\right)(\mathbf{b},\mathbf{u},\theta)&=\frac{1}{(2\pi)^{n/2}}\left\langle f(\mathbf{x}) ,\psi_{M,\mathbf{b},\mathbf{u}}^{\theta}\right\rangle \\
&=\frac{1}{(2\pi)^{n/2}}\left\langle L_M(f(\mathbf{x}))(\mathbf{u}) ,L_M(\psi_{M,\mathbf{b},\mathbf{u}}^{\theta})(\mathbf{u})\right\rangle \\
&=\frac{1}{(2\pi)^{n/2}}\mathbf{e}^{-\mathbf{i}_\mathbf{n}\mathbf{u}\cdot\mathbf{b}}\int_{\mathbb{R}^n}F_M(\mathbf{u})\overline{F_M(\psi(\mathbf{x})\mathbf{e}^{\mathbf{i}_\mathbf{n}\mathbf{x}})}\quad \overline{K_M(\mathbf{b},\mathbf{u})}d\mathbf{u}. \\
\end{aligned}
\end{equation}

Therefore, Theorem \ref{thm1} is proved.
\end{proof}

Equation (\ref{eq23}) shows that CLCST of different scales is equivalent to a set of linear canonical domain low-pass filters to process the signal. This shows that CLCST not only breaks through the limitation of traditional CST that only analyzes in the time and frequency domain, but also overcomes the defect that CLCT cannot characterize the local characteristics of the signal.

Furthermore, the CLCST defined in (\ref{eq20}) can be rewritten as:
\begin{equation}\label{eq25}
\begin{aligned}
\left(CLC\mathbb{S}_\psi^{M} f\right)(\mathbf{b},\mathbf{u},\theta)
&=\frac{\left|\det A_\mathbf{u}\right|}{(2\pi)^{n/2}} \int_{\mathbb{R}^n}f(\mathbf{x}) \overline{\psi\left(\mathbf{R}_{-\theta}\mathbf{A}_\mathbf{u}(\mathbf{x}-\mathbf{b})\right)} \mathbf{e}^{-\mathbf{i}_\mathbf{n}(\mathbf{x}\cdot\mathbf{u}+\frac{A}{2B}\mathbf{b}^2-\frac{A}{2B}\mathbf{x}^2)}d\mathbf{x}\\
&=\frac{\left|\det A_\mathbf{u}\right|}{(2\pi)^{n/2}}\mathbf{e}^{-\frac{\mathbf{i}_\mathbf{n}}{2}\mathbf{b}^2\frac{A}{2B}}\int_{\mathbb{R}^n}f(\mathbf{x})\mathbf{e}^{\frac{\mathbf{i}_\mathbf{n}}{2}\mathbf{x}^2\frac{A}{2B}}\overline{\psi\left(\mathbf{R}_{-\theta}\mathbf{A}_\mathbf{u}(\mathbf{x}-\mathbf{b})\right)}\mathbf{e}^{\mathbf{i}_\mathbf{n}\mathbf{u}\cdot\mathbf{x}}d\mathbf{x}.\\
\end{aligned}
\end{equation}

It is obvious that the proposed CLCST can be decomposed into the following three steps:

(1).The function is multiplied by a chirp signal, that is $f(\mathbf{x})\to\hat{f}(\mathbf{x})=f(\mathbf{x})\mathbf{e}^{\frac{\mathbf{i}_\mathbf{n}}{2}\mathbf{x}^2\frac{A}{2B}}.$

(2).Do the traditional CST, that is $\hat{f}(\mathbf{x})\to\left(\mathbb{S}_\psi \hat{f}\right)(\mathbf{b},\mathbf{u},\theta).$

(3).Then multiply it with another chirp signal, that is 
\begin{equation}\nonumber
\left(\mathbb{S}_\psi \hat{f}\right)(\mathbf{b},\mathbf{u},\theta)\to\left(CLC\mathbb{S}_\psi^{M} \hat{f}\right)(\mathbf{b},\mathbf{u},\theta)=\mathbf{e}^{-\frac{\mathbf{i}_\mathbf{n}}{2}\mathbf{b}^2\frac{A}{2B}}\left(\mathbb{S}_\psi \hat{f}\right)(\mathbf{b},\mathbf{u},\theta).
\end{equation}

It can be seen that the computational complexity of CLCST mainly depends on the CST operation, whose computational complexity is $O(N^{2}\mathrm{log}N)$ ($N$ is the data length). Therefore, the computational complexity of CLCST is $O(N^{2}\mathrm{log}N)$. In addition, it can be seen from Fig. \ref{fig2} that the proposed CLCST can be decomposed using the traditional CST included in its definition.

\begin{figure}
\centering
\includegraphics[width=0.9\textwidth]{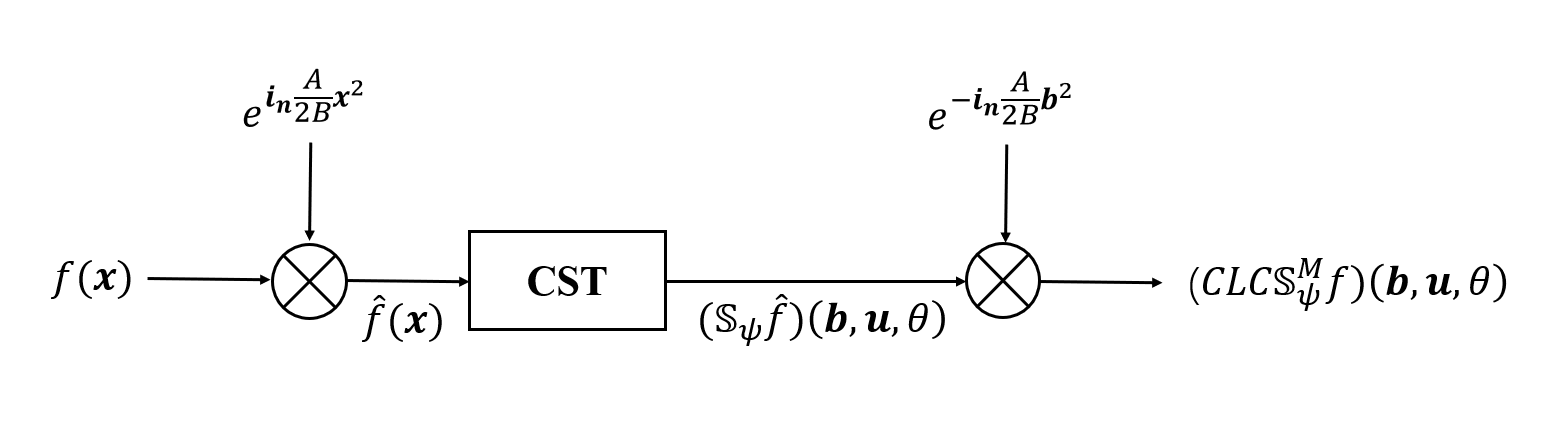}
\caption{The computational structure of the CLCST.}
\label{fig2}
\end{figure}

\subsection{Properties}
\label{subsec4}

\begin{theorem}\label{thm2}
For any $f,g \in L^2(\mathbb{R}^n, C\ell_n)$, $\alpha,\beta \in \mathbb{H}$, $\lambda \in \mathbb{R}$ and $k \in \mathbb{R}^n$, the
Clifford-valued linear canonical Stockwell transform (\ref{eq20}) possesses the following properties:

(1).Linearity

If $h=\alpha f+\beta g$, and $f \leftrightarrow \left(CLC\mathbb{S}_\psi^{M} f\right)(\mathbf{b},\mathbf{u},\theta),g \leftrightarrow \left(CLC\mathbb{S}_\psi^{M} g\right)(\mathbf{b},\mathbf{u},\theta)$, then
\begin{equation}\nonumber
\left(CLC\mathbb{S}_\psi^{M} h\right)(\mathbf{b},\mathbf{u},\theta)=\alpha\left(CLC\mathbb{S}_\psi^{M} f\right)(\mathbf{b},\mathbf{u},\theta)+\beta\left(CLC\mathbb{S}_\psi^{M} g\right)(\mathbf{b},\mathbf{u},\theta).
\end{equation}

(2).Anti-linearity

If $\psi=\alpha\psi_1+\beta\psi_2$, then
\begin{equation}\nonumber
\left(CLC\mathbb{S}_\psi^{M} f\right)(\mathbf{b},\mathbf{u},\theta)=\left(CLC\mathbb{S}_{\psi_1}^{M} f\right)(\mathbf{b},\mathbf{u},\theta)\overline{\alpha}+\left(CLC\mathbb{S}_{\psi_2}^{M} f\right)(\mathbf{b},\mathbf{u},\theta)\overline{\beta}.
\end{equation}

(3).Translation Covariance

If $f \leftrightarrow \left(CLC\mathbb{S}_\psi^{M} f\right)(\mathbf{b},\mathbf{u},\theta)$, then
\begin{equation}\nonumber
\begin{aligned}
&\left(CLC\mathbb{S}_\psi^{M} f(\mathbf{x}-\mathbf{k})\right)(\mathbf{b},\mathbf{u},\theta)\\& 
\begin{aligned}
=\frac{\left|\det A_\mathbf{u}\right|}{(2\pi)^{n/2}} \int_{\mathbb{R}^n}f(\mathbf{x}-\mathbf{k}) \overline{\psi\left(\mathbf{R}_{-\theta}\mathbf{A}_\mathbf{u}(\mathbf{x}-\mathbf{b})\right)} \mathbf{e}^{-\mathbf{i}_\mathbf{n}(\mathbf{x}\cdot\mathbf{u}+\frac{A}{2B}\mathbf{b}^2-\frac{A}{2B}\mathbf{x}^2)}d\mathbf{x}
\end{aligned} \\
&=\frac{\left|\det A_\mathbf{u}\right|}{(2\pi)^{n/2}} \int_{\mathbb{R}^n}f(\mathbf{z}) \overline{\psi\left(\mathbf{R}_{-\theta}\mathbf{A}_\mathbf{u}(\mathbf{z}-(\mathbf{b}-\mathbf{k}))\right)} \\
&\times \mathbf{e}^{-\mathbf{i}_\mathbf{n}((\mathbf{z}+\mathbf{k})\cdot\mathbf{u}+\frac{A}{2B}\mathbf{b}^2-\frac{A}{2B}(\mathbf{z}+\mathbf{k})^2)}d\mathbf{z} \\
&=\frac{\left|\det A_\mathbf{u}\right|}{(2\pi)^{n/2}}\mathbf{e}^{-\mathbf{i}_\mathbf{n}\mathbf{u}\cdot\mathbf{k}+\mathbf{i}_\mathbf{n}(\mathbf{k}^2-\mathbf{k}\mathbf{b})\frac{A}{B}} \int_{\mathbb{R}^n}f(\mathbf{z})\mathbf{e}^{\mathbf{i}_\mathbf{n}\mathbf{k}\cdot\mathbf{z}\frac{A}{B}} \overline{\psi\left(\mathbf{R}_{-\theta}\mathbf{A}_\mathbf{u}(\mathbf{z}-(\mathbf{b}-\mathbf{k}))\right)} \\
&\times \mathbf{e}^{-\mathbf{i}_\mathbf{n}(\mathbf{z}\cdot\mathbf{u}+\frac{A}{2B}(\mathbf{b}-\mathbf{k})^2-\frac{A}{2B}\mathbf{z}^2)}d\mathbf{z} \\
&=\mathbf{e}^{-\mathbf{i}_\mathbf{n}\mathbf{u}\cdot\mathbf{k}+\mathbf{i}_\mathbf{n}(\mathbf{k}^2-\mathbf{k}\mathbf{b})\frac{A}{B}}\left(CLC\mathbb{S}_\psi^{M} (f(\mathbf{z})\mathbf{e}^{\mathbf{i}_\mathbf{n}\mathbf{k}\cdot\mathbf{z}\frac{A}{B}})\right)(\mathbf{b}-\mathbf{k},\mathbf{u},\theta).
\end{aligned}
\end{equation}

(4).Dilation Covariance

If $f \leftrightarrow \left(CLC\mathbb{S}_\psi^{M} f\right)(\mathbf{b},\mathbf{u},\theta)$, then
\begin{align*}
&\left(CLC\mathbb{S}_\psi^{M} f(\lambda\mathbf{x})\right)(\mathbf{b},\mathbf{u},\theta)\\& 
\begin{aligned}
=\frac{\left|\det A_\mathbf{u}\right|}{(2\pi)^{n/2}} \int_{\mathbb{R}^n}f(\lambda\mathbf{x}) \overline{\psi\left(\mathbf{R}_{-\theta}\mathbf{A}_\mathbf{u}(\mathbf{x}-\mathbf{b})\right)} \mathbf{e}^{-\mathbf{i}_\mathbf{n}(\mathbf{x}\cdot\mathbf{u}+\frac{A}{2B}\mathbf{b}^2-\frac{A}{2B}\mathbf{x}^2)}d\mathbf{x}
\end{aligned} \\
&=\frac{1}{\lambda^n}\frac{\left|\det A_\mathbf{u}\right|}{(2\pi)^{n/2}} \int_{\mathbb{R}^n}f(\mathbf{x}') \overline{\psi\left(\mathbf{R}_{-\theta}\mathbf{A}_\frac{\mathbf{u}}{\lambda}(\mathbf{x}'-\mathbf{b}\lambda)\right)} \\
&\times \mathbf{e}^{-\mathbf{i}_\mathbf{n}(\frac{\mathbf{x}'}{\lambda}\cdot\mathbf{u}+\frac{A}{2B\lambda^2}(\mathbf{b}\lambda)^2-\frac{A}{2B}(\frac{\mathbf{x}'}{\lambda})^2)}d\mathbf{x}' \\
&=\frac{1}{\lambda^n}\left(CLC\mathbb{S}_\psi^{M'} f\right)(\mathbf{b}\lambda,\frac{\mathbf{u}}{\lambda},\theta),
\end{align*}

where $M'=(A,\lambda^2B,C,D)$.

(5).Parity

If $f \leftrightarrow \left(CLC\mathbb{S}_\psi^{M} f\right)(\mathbf{b},\mathbf{u},\theta)$, then
\begin{align*}
&\left(CLC\mathbb{S}_\psi^{M} f(-\mathbf{x})\right)(\mathbf{b},\mathbf{u},\theta)\\& 
\begin{aligned}
=\frac{\left|\det A_\mathbf{u}\right|}{(2\pi)^{n/2}} \int_{\mathbb{R}^n}f(-\mathbf{x}) \overline{\psi\left(\mathbf{R}_{-\theta}\mathbf{A}_\mathbf{u}(\mathbf{x}-\mathbf{b})\right)} \mathbf{e}^{-\mathbf{i}_\mathbf{n}(\mathbf{x}\cdot\mathbf{u}+\frac{A}{2B}\mathbf{b}^2-\frac{A}{2B}\mathbf{x}^2)}d\mathbf{x}
\end{aligned} \\
&=(-1)^n\frac{\left|\det A_\mathbf{u}\right|}{(2\pi)^{n/2}} \int_{\mathbb{R}^n}f(\mathbf{x}') \overline{\psi\left(\mathbf{R}_{-\theta}\mathbf{A}_{-\mathbf{u}}(\mathbf{x}'-(-\mathbf{b}))\right)} \\
&\times \mathbf{e}^{-\mathbf{i}_\mathbf{n}(\mathbf{x}'\cdot(-\mathbf{u})+\frac{A}{2B}(-\mathbf{b})^2-\frac{A}{2B}\mathbf{x}'^2)}d\mathbf{x}' \\
&=(-1)^n\left(CLC\mathbb{S}_\psi^{M} f\right)(-\mathbf{b},-\mathbf{u},\theta).
\end{align*}
\end{theorem}

\begin{lemma}\label{lem1}
The spectral form of the proposed transform in (\ref{eq20}) can be expressed as follows:
\begin{equation}\label{eq26}
\begin{aligned}
\left(CLC\mathbb{S}_\psi^{M} f\right)(\mathbf{b},\mathbf{u},\theta)
&=\frac{\left|\det A_\mathbf{u}\right|}{\left(2\pi\right)^n}\int_{\mathbb{R}^n}\mathcal{F}_{C\ell}\big[f(\mathbf{x})\mathbf{e}^{\mathbf{i}_\mathbf{n}\frac{A}{2B}\mathbf{x}^2}\big](\mathbf{w}) \mathbf{e}^{\mathbf{i}_\mathbf{n}(\mathbf{w}\cdot\mathbf{b}-\frac{A}{2B}\mathbf{b}^2)}\\
&\times \overline{\mathcal{F}_{C\ell}\Big[\mathbf{e}^{\mathbf{i}_\mathbf{n}\mathbf{u}\cdot\mathbf{y}}\psi\big(R_{-\theta}A_\mathbf{u}\mathbf{y}\big)\Big](\mathbf{w})}\mathbf{e}^{-\mathbf{i}_\mathbf{n}\mathbf{u}\cdot\mathbf{b}} d\mathbf{w}.
\end{aligned}
\end{equation}
\end{lemma}

\begin{theorem}\label{thm3}
Let $\left(CLC\mathbb{S}_\psi^{M} f\right)(\mathbf{b},\mathbf{u},\theta)$ and $\left(CLC\mathbb{S}_\psi^{M} g\right)(\mathbf{b},\mathbf{u},\theta)$ represent the Clifford-valued linear canonical Stockwell transforms of the functions $f$ and $g$, respectively. Then, the following relation holds:
\begin{equation}\label{eq27}
\begin{aligned}
&\left\langle\left(CLC\mathbb{S}_\psi^{M} f\right), \left(CLC\mathbb{S}_\psi^{M} g\right)\right\rangle_{L^2(\mathbb{R}^n\times\mathbb{R}^n\times SO(n),C\ell_n)}\\
&=\left\langle f\mathbf{e}^{\mathbf{i}_\mathbf{n}\frac{A}{2B}\mathbf{x}^2} C_\psi, g\mathbf{e}^{\mathbf{i}_\mathbf{n}\frac{A}{2B}\mathbf{x}^2}\right\rangle_{L^2(\mathbb{R}^n,C\ell_n)},
\end{aligned}
\end{equation}
where
\begin{equation}\label{eq28}
\mathcal{C}_{\psi}=\frac{1}{(2\pi)^{n}}\int_{\mathbb{R}^{n}\times SO(n)}\left|\mathcal{F}_{C\ell}\left[\mathbf{e}^{\mathbf{i}_{\mathbf{n}}\mathbf{u}\cdot\mathbf{y}}\psi\Big(R_{-\theta}A_{\mathbf{u}}\mathbf{y}\Big)\right](\mathbf{w})\right|^{2}\frac{d\theta d\mathbf{u}}{\left|\det A_{\mathbf{u}}\right|^{-2}}<\infty.
\end{equation}
\end{theorem}

\begin{proof}
Using the spectral representation (\ref{eq26}) of the proposed transform, we derive:
\begin{align*}
&\left\langle(CLC\mathbb{S}_\psi^{M} f), (CLC\mathbb{S}_\psi^{M} g)\right\rangle_{L^2(\mathbb{R}^n\times\mathbb{R}^n\times\mathrm{SO}(n),C\ell_n)}\\
&=\int_{\mathbb{R}^n\times\mathbb{R}^n\times\mathrm{SO}(n)}\left(CLC\mathbb{S}_\psi^{M} f\right)(\mathbf{b},\mathbf{u},\theta)\overline{\left(CLC\mathbb{S}_\psi^{M} g\right)(\mathbf{b},\mathbf{u},\theta)} d\mathbf{w} d\mathbf{u} d\theta\\
&=\frac{1}{\left(2\pi\right)^{2n}}\int_{\mathbb{R}^{n}\times\mathbb{R}^{n}\times\mathrm{SO}(n)}\int_{\mathbb{R}^{n}}\mathcal{F}_{C\ell}\big[f\mathbf{e}^{\mathbf{i}_\mathbf{n}\frac{A}{2B}\mathbf{x}^2}\big](\mathbf{w}) \mathbf{e}^{\mathbf{i}_\mathbf{n}(\mathbf{w}\cdot\mathbf{b}-\frac{A}{2B}\mathbf{b}^2)} \overline{\mathcal{F}_{C\ell}\bigg[\mathbf{e}^{\mathbf{i}_{\mathbf{n}}\mathbf{u}\cdot\mathbf{y}}\psi\bigg(R_{-\theta}A_{\mathbf{u}}\mathbf{y}\bigg)\bigg](\mathbf{w})} \\
&\times \mathbf{e}^{-\mathbf{i}_{\mathbf{n}}\mathbf{u}\cdot\mathbf{b}} \frac{d\mathbf{w}}{\left|\det A_{\mathbf{u}}\right|^{-1}}\int_{\mathbb{R}^{n}} \mathbf{e}^{\mathbf{i}_{\mathbf{n}}\mathbf{u}\cdot\mathbf{b}}\mathcal{F}_{C\ell}\Big[\mathbf{e}^{\mathbf{i}_{\mathbf{n}}\mathbf{u}\cdot\mathbf{y}}\psi\Big(R_{-\theta}A_{\mathbf{u}}\mathbf{y}\Big)\Big](\mathbf{w}') \\
&\times\overline{\mathcal{F}_{C\ell}[g\mathbf{e}^{\mathbf{i}_\mathbf{n}\frac{A}{2B}\mathbf{x}^2}](\mathbf{w}')}\mathbf{e}^{-\mathbf{i}_\mathbf{n}(\mathbf{w'}\cdot\mathbf{b}-\frac{A}{2B}\mathbf{b}^2)} \frac{d\mathbf{w}'}{\left|\det A_{\mathbf{u}}\right|^{-1}} d\mathbf{w} d\mathbf{u} d\theta \\
&=\frac1{\left(2\pi\right)^{2n}}\int_{\mathbb{R}^n\times\mathbb{R}^n\times\mathrm{SO}(n)}\int_{\mathbb{R}^n}\int_{\mathbb{R}^n}\mathcal{F}_{C\ell}\big[f\mathbf{e}^{\mathbf{i}_\mathbf{n}\frac{A}{2B}\mathbf{x}^2}\big](\mathbf{w}) \mathbf{e}^{\mathbf{i}_\mathbf{n}(\mathbf{w}\cdot\mathbf{b}-\frac{A}{2B}\mathbf{b}^2)} \\
&\times\overline{\mathcal{F}_{C\ell}\bigg[\mathbf{e}^{\mathbf{i}_{\mathbf{n}}\mathbf{u}\cdot\mathbf{y}}\psi\bigg(R_{-\theta}A_{\mathbf{u}}\mathbf{y}\bigg)\bigg](\mathbf{w})}\mathbf{e}^{-\mathbf{i}_{\mathbf{n}}\mathbf{u}\cdot\mathbf{b}}\mathbf{e}^{\mathbf{i}_{\mathbf{n}}\mathbf{u}\cdot\mathbf{b}}\mathcal{F}_{C\ell}\Big[\mathbf{e}^{\mathbf{i}_{\mathbf{n}}\mathbf{u}\cdot\mathbf{y}}\psi\Big(R_{-\theta}A_{\mathbf{u}}\mathbf{y}\Big)\Big](\mathbf{w}^{\prime}) \\
&\times \mathbf{e}^{-\mathbf{i}_\mathbf{n}(\mathbf{w}'\cdot\mathbf{b}-\frac{A}{2B}\mathbf{b}^2)} \overline{\mathcal{F}_{C\ell}[g\mathbf{e}^{\mathbf{i}_\mathbf{n}\frac{A}{2B}\mathbf{x}^2}](\mathbf{w}')} d\mathbf{w} d\mathbf{w}' \frac{d\mathbf{w} d\mathbf{u} d\theta}{\left|\det A_\mathbf{u}\right|^{-2}} \\
&=\frac1{\left(2\pi\right)^{2n}}\int_{\mathbb{R}^n\times\mathbb{R}^n\times\mathrm{SO}(n)}\int_{\mathbb{R}^n}\int_{\mathbb{R}^n}\mathcal{F}_{C\ell}\big[f\mathbf{e}^{\mathbf{i}_\mathbf{n}\frac{A}{2B}\mathbf{x}^2}\big](\mathbf{w}) \mathbf{e}^{\mathbf{i}_\mathbf{n}(\mathbf{w}\cdot\mathbf{b}-\frac{A}{2B}\mathbf{b}^2)} \\
&\times\overline{\mathcal{F}_{C\ell}\Big[\mathbf{e}^{\mathbf{i}_{\mathbf{n}}\mathbf{u}\cdot\mathbf{y}}\psi\Big(R_{-\theta}A_{\mathbf{u}}\mathbf{y}\Big)\Big](\mathbf{w})} \mathcal{F}_{C\ell}\Big[\mathbf{e}^{\mathbf{i}_{\mathbf{n}}\mathbf{u}\cdot\mathbf{y}}\psi\Big(R_{-\theta}A_{\mathbf{u}}\mathbf{y}\Big)\Big](\mathbf{w}') \\
&\times \mathbf{e}^{-\mathbf{i}_\mathbf{n}(\mathbf{w}'\cdot\mathbf{b}-\frac{A}{2B}\mathbf{b}^2)}\overline{\mathcal{F}_{C\ell}[g\mathbf{e}^{\mathbf{i}_\mathbf{n}\frac{A}{2B}\mathbf{x}^2}](\mathbf{w}')} d\mathbf{w} d\mathbf{w}' \frac{d\mathbf{w} d\mathbf{u} d\theta}{\left|\det A_{\mathbf{u}}\right|^{-2}}.
\end{align*}

By applying Fubini's theorem, we obtain an equivalent expression as follows:
\begin{align*}
&\left\langle(CLC\mathbb{S}_\psi^{M} f), (CLC\mathbb{S}_\psi^{M} g)\right\rangle_{L^2(\mathbb{R}^n\times\mathbb{R}^n\times\mathrm{SO}(n),C\ell_n)}\\
&=\frac{1}{\left(2\pi\right)^{2n}}\int_{\mathbb{R}^{n}\times\mathbb{R}^{n}\times\mathrm{SO}(n)}\int_{\mathbb{R}^{n}}\int_{\mathbb{R}^{n}}\mathcal{F}_{C\ell}\big[f\mathbf{e}^{\mathbf{i}_\mathbf{n}\frac{A}{2B}\mathbf{x}^2}\big](\mathbf{w}) \mathbf{e}^{\mathbf{i}_{\mathbf{n}}\mathbf{w}\cdot\mathbf{b}} \mathbf{e}^{-\mathbf{i}_{\mathbf{n}}\mathbf{w}^{\prime}\cdot\mathbf{b}}\\
&\times\overline{\mathcal{F}_{C\ell}\Big[\mathbf{e}^{\mathbf{i}_{\mathbf{n}}\mathbf{u}\cdot\mathbf{y}}\psi\Big(R_{-\theta}A_{\mathbf{u}}\mathbf{y}\Big)\Big](\mathbf{w})} \mathcal{F}_{C\ell}\Big[\mathbf{e}^{\mathbf{i}_{\mathbf{n}}\mathbf{u}\cdot\mathbf{y}}\psi\Big(R_{-\theta}A_{\mathbf{u}}\mathbf{y}\Big)\Big](\mathbf{w}')\\
&\times\overline{\mathcal{F}_{C\ell}[g\mathbf{e}^{\mathbf{i}_\mathbf{n}\frac{A}{2B}\mathbf{x}^2}](\mathbf{w}^{\prime})} d\mathbf{w} d\mathbf{w}^{\prime} \frac{d\mathbf{w} d\mathbf{u} d\theta}{\left|\det A\mathbf{u}\right|^{-2}}\\
&=\frac{1}{\left(2\pi\right)^{n}}\int_{\mathbb{R}^{n}\times\mathrm{SO(n)}}\int_{\mathbb{R}^{n}}\int_{\mathbb{R}^{n}}\mathcal{F}_{C\ell}\big[f\mathbf{e}^{\mathbf{i}_\mathbf{n}\frac{A}{2B}\mathbf{x}^2}\big](\mathbf{w}) \delta\big(\mathbf{w}-\mathbf{w}^{\prime}\big) \\
&\times\overline{\mathcal{F}_{C\ell}\Big[\mathbf{e}^{\mathbf{i}_{\mathbf{n}}\mathbf{u}\cdot\mathbf{y}}\psi\Big(R_{-\theta}A_{\mathbf{u}}\mathbf{y}\Big)\Big](\mathbf{w})} \mathcal{F}_{C\ell}\Big[\mathbf{e}^{\mathbf{i}_{\mathbf{n}}\mathbf{u}\cdot\mathbf{y}}\psi\Big(R_{-\theta}A_{\mathbf{u}}\mathbf{y}\Big)\Big](\mathbf{w}')\\
&\times\overline{\mathcal{F}_{C\ell}[g\mathbf{e}^{\mathbf{i}_\mathbf{n}\frac{A}{2B}\mathbf{x}^2}](\mathbf{w}^{\prime})} d\mathbf{w} d\mathbf{w}^{\prime} \frac{ d\theta d\mathbf{u}}{\left|\det A\mathbf{u}\right|^{-2}}\\
&=\frac{1}{\left(2\pi\right)^{n}}\int_{\mathbb{R}^{n}\times\mathrm{SO(n)}}\int_{\mathbb{R}^{n}}\mathcal{F}_{C\ell}\big[f\mathbf{e}^{\mathbf{i}_\mathbf{n}\frac{A}{2B}\mathbf{x}^2}\big](\mathbf{w}) \overline{{\mathcal{F}_{C\ell}\Big[\mathbf{e}^{\mathbf{i}_{\mathbf{n}}\mathbf{u}\cdot\mathbf{y}}\psi\left(R_{-\theta}A_{\mathbf{u}}\mathbf{y}\right)\Big](\mathbf{w})}} \\
&\times\mathcal{F}_{C\ell}\Big[\mathbf{e}^{\mathbf{i}_{\mathbf{n}}\mathbf{u}\cdot\mathbf{y}}\psi\left(R_{-\theta}A_{\mathbf{u}}\mathbf{y}\right)\Big](\mathbf{w}) \overline{\mathcal{F}_{C\ell}[g\mathbf{e}^{\mathbf{i}_\mathbf{n}\frac{A}{2B}\mathbf{x}^2}](\mathbf{w})} d\mathbf{w} \frac{d\theta d\mathbf{u}}{\left|\det A_{\mathbf{u}}\right|^{-2}} \\
&=\frac{1}{\left(2\pi\right)^{n}}\int_{\mathbb{R}^{n}}\mathcal{F}_{C\ell}\big[f\mathbf{e}^{\mathbf{i}_\mathbf{n}\frac{A}{2B}\mathbf{x}^2}\big](\mathbf{w})\int_{\mathbb{R}^{n}\times\mathrm{SO}(\mathrm{n})}\Big|\mathcal{F}_{C\ell}\Big[\mathbf{e}^{\mathbf{i}_{\mathbf{n}}\mathbf{u}\cdot\mathbf{y}}\psi\Big(R_{-\theta}A_{\mathbf{u}}\mathbf{y}\Big)\Big](\mathbf{w})\Big|^{2} \\
&\times\frac{d\theta d\mathbf{u}}{\left|\det A_{\mathbf{u}}\right|^{-2}}\overline{\mathcal{F}_{C\ell}[g\mathbf{e}^{\mathbf{i}_\mathbf{n}\frac{A}{2B}\mathbf{x}^2}](\mathbf{w})} d\mathbf{w} \\
&=\left\langle\mathcal{F}_{C\ell}\big[f\mathbf{e}^{\mathbf{i}_\mathbf{n}\frac{A}{2B}\mathbf{x}^2}\big] C_{\psi}, \mathcal{F}_{C\ell}\big[g\mathbf{e}^{\mathbf{i}_\mathbf{n}\frac{A}{2B}\mathbf{x}^2}\big]\right\rangle_{L^{2}(\mathbb{R}^{n},C\ell_{n})} \\
&=\left\langle f\mathbf{e}^{\mathbf{i}_\mathbf{n}\frac{A}{2B}\mathbf{x}^2} C_{\psi}, g\mathbf{e}^{\mathbf{i}_\mathbf{n}\frac{A}{2B}\mathbf{x}^2}\right\rangle_{L^{2}(\mathbb{R}^{n}, C\ell_{n})}.
\end{align*}
\end{proof}

\begin{corollary}\label{cor1}
When $f = g$ and $C_{\psi}$ is a real-valued constant, the orthogonality relation (\ref{eq27}) simplifies to:
\begin{equation}\label{eq29}
\left\|\left(CLC\mathbb{S}_\psi^M f\right)(\mathbf{b},\mathbf{u},\theta)\right\|_{L^2(\mathbb{R}^n\times\mathbb{R}^n\times SO(n),C\ell_n)}^2=C_\psi\left\|f\mathbf{e}^{\mathbf{i}_\mathbf{n}\frac{A}{2B}\mathbf{x}^2}
\right\|_{L^2(\mathbb{R}^n,C\ell_n)}^2.
\end{equation}
\end{corollary}

\begin{remark}
For any window function $\psi \in L^2(\mathbb{R}^n, C\ell_n)
$ with $C_{\psi} = 1$, the Clifford-valued linear canonical Stockwell transform defined in (\ref{eq20}) acts as an isometry, mapping the signal space $L^2(\mathbb{R}^n, C\ell_n)$ onto the transform space $L^2(\mathbb{R}^n\times\mathbb{R}^n\times SO(n),C\ell_n)$.
\end{remark}

The following theorem ensures that the original Clifford-valued signal can be reconstructed from its Clifford-valued linear canonical Stockwell transform, as defined in (\ref{eq20}).

\begin{theorem}\label{thm4}
Any square-integrable Clifford-valued function $f$ can be recovered from its Clifford-valued linear canonical Stockwell transform $\left(CLC\mathbb{S}_{\psi}^{M} f\right)(\mathbf{b},\mathbf{u},\theta)$ using the following reconstruction formula:
\begin{equation}\label{eq30}
f(\mathbf{x})=\frac{1}{(2\pi)^{n/2}} \int_{\mathbb{R}^n\times\mathbb{R}^n\times SO(n)}\left(CLC\mathbb{S}_\psi^M f\right)(\mathbf{b},\mathbf{u},\theta) \psi_{M,\mathbf{b},\mathbf{u}}^\theta(\mathbf{x}) C_\psi^{-1} d\mathbf{w} d\mathbf{u} d\theta,\quad a.e.
\end{equation}
\end{theorem}

\begin{proof}
The orthogonality relation (\ref{eq27}) implies that for any $g \in L^2(\mathbb{R}^n, C\ell_n)$,
\begin{align*}
&\left\langle f\mathbf{e}^{\mathbf{i}_\mathbf{n}\frac{A}{2B}\mathbf{x}^2} C_\psi,g\mathbf{e}^{\mathbf{i}_\mathbf{n}\frac{A}{2B}\mathbf{x}^2}\right\rangle_{L^2(\mathbb{R}^n,C\ell_n)} \\
&=\int_{\mathbb{R}^n\times\mathbb{R}^n\times\mathrm{SO}(n)}\left(CLC\mathbb{S}_\psi^M f\right)(\mathbf{b},\mathbf{u},\theta) \overline{\left(CLC\mathbb{S}_\psi^M g\right)(\mathbf{b},\mathbf{u},\theta)} d\mathbf{w} d\mathbf{u} d\theta \\
&=\frac{1}{(2\pi)^{n/2}} \int_{\mathbb{R}^{n}\times\mathbb{R}^{n}\times\mathrm{SO}(n)}\big(CLC\mathbb{S}_{\psi}^Mf\big)(\mathbf{b},\mathbf{u},\theta)\left\{\overline{\int_{\mathbb{R}^{n}}g(\mathbf{x}) \overline{\psi_{M,\mathbf{b},\mathbf{u}}^{\theta}(\mathbf{x})} \mathbf{d}\mathbf{x}}\right\} d\mathbf{w} d\mathbf{u} d\theta \\
&=\frac{1}{(2\pi)^{n/2}} \int_{\mathbb{R}^{n}\times\mathbb{R}^{n}\times\mathrm{SO}(n)}\int_{\mathbb{R}^{n}}\left(CLC\mathbb{S}_{\psi}^Mf\right)(\mathbf{b},\mathbf{u},\theta) \psi_{M,\mathbf{b},\mathbf{u}}^{\theta}(\mathbf{x}) \mathbf{e}^{\mathbf{i}_\mathbf{n}\frac{A}{2B}\mathbf{x}^2}\overline{g(\mathbf{x})\mathbf{e}^{\mathbf{i}_\mathbf{n}\frac{A}{2B}\mathbf{x}^2}} d\mathbf{x} d\mathbf{w} d\mathbf{u} d\theta \\
&=\frac{1}{(2\pi)^{n/2}} \int_{\mathbb{R}^{n}}\int_{\mathbb{R}^{n}\times\mathbb{R}^{n}\times\mathrm{SO}(n)}\left(CLC\mathbb{S}_{\psi}^Mf\right)(\mathbf{b},\mathbf{u},\theta) \psi_{M,\mathbf{b},\mathbf{u}}^{\theta}(\mathbf{x}) \mathbf{e}^{\mathbf{i}_\mathbf{n}\frac{A}{2B}\mathbf{x}^2}d\mathbf{w} d\mathbf{u} d\theta \overline{g(\mathbf{x})\mathbf{e}^{\mathbf{i}_\mathbf{n}\frac{A}{2B}\mathbf{x}^2}} d\mathbf{x} \\
&=\frac{1}{(2\pi)^{n/2}} \left\langle\int_{\mathbb{R}^{n}\times\mathbb{R}^{n}\times\mathrm{SO}(n)}\left(CLC\mathbb{S}_{\psi}^Mf\right)(\mathbf{b},\mathbf{u},\theta) \psi_{M,\mathbf{b},\mathbf{u}}^{\theta}(\mathbf{x})\mathbf{e}^{\mathbf{i}_\mathbf{n}\frac{A}{2B}\mathbf{x}^2} d\mathbf{w} d\mathbf{u} d\theta, g(\mathbf{x})\mathbf{e}^{\mathbf{i}_\mathbf{n}\frac{A}{2B}\mathbf{x}^2}\right\rangle.
\end{align*}
Since $g \in L^2(\mathbb{R}^n, C\ell_n)$ is arbitrarily, so we obtain the desired result.
\end{proof}

The following theorem provides an alternative reconstruction formula for the Clifford-valued Stockwell transform (\ref{eq20}).

\begin{theorem}\label{thm5}
Let $\left(CLC\mathbb{S}_{\psi}^{M} f\right)(\mathbf{b},\mathbf{u},\theta)$ denote the Clifford-valued linear canonical Stockwell transform of any Clifford-valued function $f \in L^2(\mathbb{R}^n, C\ell_n)$ with respect to the window function $\psi\in L^1(\mathbb{R}^n, C\ell_n)\cap L^2(\mathbb{R}^n,C\ell_n)$ with
\begin{equation}\label{eq31}
\int_{\mathbb{R}^n}\psi(\mathbf{x}) d\mathbf{x}=1.
\end{equation}
Then, the Clifford-valued function $f$ can reconstructed via
\begin{equation}\label{eq32}
f(\mathbf{x})=\mathbf{e}^{-\mathbf{i}_\mathbf{n}\frac{A}{2B}\mathbf{x}^2}\mathcal{F}_{C\ell}^{-1}\left[\mathbf{e}^{\mathbf{i}_\mathbf{n}\frac{A}{2B}\mathbf{b}^2}\int_{\mathbb{R}^n}\left(CLC\mathbb{S}_\psi^M f\right)(\mathbf{b},\mathbf{u},\theta)\frac{d\mathbf{b}}{\left|\det A_\mathbf{u}\right|^{1-n}}\right](\mathbf{u}).
\end{equation}
\end{theorem}

\begin{proof}
Applying Definition \ref{df4} in conjunction with Fubini's theorem, we derive
\begin{equation}\nonumber
\begin{aligned}
&\mathbf{e}^{\mathbf{i}_\mathbf{n}\frac{A}{2B}\mathbf{b}^2}\int_{\mathbb{R}^{n}}\left(CLC\mathbb{S}_{\psi}^Mf\right)(\mathbf{b},\mathbf{u},\theta)\frac{d\mathbf{b}}{\left|\det A_{\mathbf{u}}\right|^{1-n}} \\
&=\frac{\mathbf{e}^{\mathbf{i}_\mathbf{n}\frac{A}{2B}\mathbf{b}^2}}{(2\pi)^{n/2}} \int_{\mathbb{R}^{n}}\int_{\mathbb{R}^{n}}f(\mathbf{x}) \overline{{\psi_{M,\mathbf{b},\mathbf{u}}^{\theta}(\mathbf{x})}} d\mathbf{x} \frac{d\mathbf{b}}{\left|\det A_{\mathbf{u}}\right|^{1-n}} \\
&=\frac{\mathbf{e}^{\mathbf{i}_\mathbf{n}\frac{A}{2B}\mathbf{b}^2}}{(2\pi)^{n/2}} \int_{\mathbb{R}^{n}}\int_{\mathbb{R}^{n}}f(\mathbf{x}) \overline{\psi\Big(R_{-\theta}A_{\mathbf{u}}(\mathbf{x}-\mathbf{b})\Big)} \\
&\times\mathbf{e}^{-\mathbf{i}_\mathbf{n}(\mathbf{x}\cdot\mathbf{u}+\frac{A}{2B}\mathbf{b}^2-\frac{A}{2B}\mathbf{x}^2)} d\mathbf{x} \frac{d\mathbf{b}}{\left|\det A_{\mathbf{u}}\right|^{-n}} \\
&=\frac{1}{(2\pi)^{n/2}} \int_{\mathbb{R}^{n}}f(\mathbf{x}) \int_{\mathbb{R}^{n}}\overline{{\psi\Big(R_{-\theta}A_{\mathbf{u}}(\mathbf{x}-\mathbf{b})\Big)}} \frac{d\mathbf{b}}{\left|\det A_{\mathbf{u}}\right|^{-n}} \\
&\times\mathbf{e}^{-\mathbf{i}_\mathbf{n}(\mathbf{x}\cdot\mathbf{u}-\frac{A}{2B}\mathbf{x}^2)} d\mathbf{x} \\
&=\frac1{(2\pi)^{n/2}}\int_{\mathbb{R}^n}f(\mathbf{x}) \int_{\mathbb{R}^n}\overline{\psi(\mathbf{y})} d\mathbf{y} \mathbf{e}^{-\mathbf{i}_\mathbf{n}(\mathbf{x}\cdot\mathbf{u}-\frac{A}{2B}\mathbf{x}^2)} d\mathbf{x} \\
&=\frac1{(2\pi)^{n/2}} \int_{\mathbb{R}^n}f(\mathbf{x}) \mathbf{e}^{-\mathbf{i}_\mathbf{n}(\mathbf{x}\cdot\mathbf{u}-\frac{A}{2B}\mathbf{x}^2)} d\mathbf{x} \\
&=\mathcal{F}_{C\ell}\left[f(\mathbf{x})\mathbf{e}^{\mathbf{i}_\mathbf{n}\frac{A}{2B}\mathbf{x}^2}\right](\mathbf{u}).
\end{aligned}
\end{equation}
Taking inverse Clifford-valued Fourier transform on both sides, we get the desired result (\ref{eq32}).
\end{proof}

The following theorem provides a full characterization of the range of the proposed transform $\left(CLC\mathbb{S}_{\psi}^{M} f\right)(\mathbf{b},\mathbf{u},\theta)$. This result is derived from the reconstruction formula (\ref{eq32}) and the application of Fubini's theorem.

\begin{theorem}\label{thm6}
If $\psi(\mathbf{x}) \in L^2(\mathbb{R}^n, C\ell_n)$ satisfies the admissibility condition (\ref{eq28}), then the range of the Clifford-valued linear canonical Stockwell transform $\left(CLC\mathbb{S}{\psi}^{M} f\right)(\mathbf{b},\mathbf{u},\theta)$ forms a reproducing kernel in $L^{2}\left(\mathbb{R}^{n}\times\mathbb{R}^{n} \times SO(n), C\ell{n}\right)$, with the kernel defined by
\begin{equation}\label{eq33}
K_{\psi}\left(\mathbf{b},\mathbf{u},\theta;\mathbf{b}',\mathbf{u}',\theta'\right)=\left\langle\psi_{M,\mathbf{b},\mathbf{u}}^{\theta}C_{\psi}^{-1}, \psi_{M,\mathbf{b}',\mathbf{u}'}^{\theta'}\right\rangle_{L^{2}(\mathbb{R}^{n},C\ell_{n})}.
\end{equation}
Additionally, the reproducing kernel $K_{\psi}$ is bounded at each point, meaning:
\begin{equation}\label{eq34}
\left|K_\psi(\mathbf{b},\mathbf{u},\theta;\mathbf{b'},\mathbf{u'},\theta')\right|\leq\left|\frac{|\det A_\mathbf{u}|^{1-n} |\det A_\mathbf{u'}|^{1-n}}{C_\psi}\right|^{1/2}\left\|\psi\right\|_{L^1(\mathbb{R}^n,C\ell_n)}.
\end{equation}
\end{theorem}

\begin{proof}
By applying the reconstruction formula (\ref{eq30}), we obtain
\begin{equation}\nonumber
\begin{aligned}
&\left(CLC\mathbb{S}_{\psi}^Mf\right)(\mathbf{b}^{\prime},\mathbf{u}^{\prime},\theta^{\prime}) \\
&=\frac1{(2\pi)^{n/2}} \int_{\mathbb{R}^n}f(\mathbf{x}) \overline{\psi_{M,\mathbf{b}^{\prime},\mathbf{u}^{\prime}}^{\theta^{\prime}}(\mathbf{x})} d\mathbf{x} \\
&=\frac{|\det A_\mathbf{u}|}{(2\pi)^{n/2}} \int_{\mathbb{R}^{n}}f(\mathbf{x}) \overline{\psi\Big(R_{-\theta'}A_{\mathbf{u'}}(\mathbf{x}-\mathbf{b'})\Big)} \mathbf{e}^{-\mathbf{i}_{\mathbf{n}} (\mathbf{x}\cdot\mathbf{u'}+\frac{A}{2B}\mathbf{b}'^2-\frac{A}{2B}\mathbf{x}^2)} d\mathbf{x} \\
&=\frac{|\det A_\mathbf{u}|}{(2\pi)^{n}} \int_{\mathbb{R}^{n}}\int_{\mathbb{R}^{n}\times\mathbb{R}^{n}\times\mathrm{SO}(n)}\left(CLC\mathbb{S}_{\psi}^M f\right)(\mathbf{b},\mathbf{u},\theta) \psi_{M,\mathbf{b},\mathbf{u}}^{\theta}(\mathbf{x}) C_{\psi}^{-1} d\mathbf{w} d\mathbf{u} d\theta \\
&\times\overline{\psi\Big(R_{-\theta'}A_{\mathbf{u}'}(\mathbf{x}-\mathbf{b}')\Big)} \mathbf{e}^{-\mathbf{i}_{\mathbf{n}} (\mathbf{x}\cdot\mathbf{u'}+\frac{A}{2B}\mathbf{b}'^2-\frac{A}{2B}\mathbf{x}^2)} d\mathbf{x} \\
&=\frac{|\det A_\mathbf{u}|}{(2\pi)^n} \int_{\mathbb{R}^n\times\mathbb{R}^n\times\mathrm{SO}(n)}\int_{\mathbb{R}^n}\left(CLC\mathbb{S}_\psi^M f\right)(\mathbf{b},\mathbf{u},\theta) \psi_{M,\mathbf{b},\mathbf{u}}^\theta(\mathbf{x}) C_\psi^{-1} \\
&\times \overline{\psi\Big(R_{-\theta'}A_{\mathbf{u'}}(\mathbf{x}-\mathbf{b'})\Big)} \mathbf{e}^{-\mathbf{i}_{\mathbf{n}} (\mathbf{x}\cdot\mathbf{u'}+\frac{A}{2B}\mathbf{b}'^2-\frac{A}{2B}\mathbf{x}^2)} d\mathbf{x} d\mathbf{w} d\mathbf{u} d\theta \\
&=\frac{|\det A_\mathbf{u}|}{(2\pi)^n} \int_{\mathbb{R}^n\times\mathbb{R}^n\times\mathrm{SO}(n)}\left(CLC\mathbb{S}_\psi^M f\right)(\mathbf{b},\mathbf{u},\theta) \int_{\mathbb{R}^n}\psi_{M,\mathbf{b},\mathbf{u}}^\theta(\mathbf{x}) C_\psi^{-1} \\
&\times \overline{\psi\Big(R_{-\theta'}A_{\mathbf{u'}}(\mathbf{x}-\mathbf{b'})\Big)} \mathbf{e}^{-\mathbf{i}_{\mathbf{n}} (\mathbf{x}\cdot\mathbf{u'}+\frac{A}{2B}\mathbf{b}'^2-\frac{A}{2B}\mathbf{x}^2)} d\mathbf{x} d\mathbf{w} d\mathbf{u} d\theta \\
&=\frac{1}{(2\pi)^{n}} \int_{\mathbb{R}^{n}\times\mathbb{R}^{n}\times\mathrm{SO}(n)} \left(CLC\mathbb{S}_{\psi}^Mf\right)(\mathbf{b},\mathbf{u},\theta)\left\langle\psi_{M,\mathbf{b},\mathbf{u}}^{\theta}C_{\psi}^{-1}, \psi_{M,\mathbf{b}^{\prime},\mathbf{u}^{\prime}}^{\theta^{\prime}}\right\rangle_{L^{2}(\mathbb{R}^{n},C\ell_{n})} d\mathbf{w}d\mathbf{u}d\theta \\
&=\frac{1}{(2\pi)^{n}}\int_{\mathbb{R}^{n}\times\mathbb{R}^{n}\times\mathrm{SO}(n)}\left(CLC\mathbb{S}_{\psi}^Mf\right)(\mathbf{b},\mathbf{u},\theta) K_{\psi}\left(\mathbf{b},\mathbf{u},\theta;\mathbf{b}^{\prime},\mathbf{u}^{\prime},\theta^{\prime}\right) d\mathbf{w} d\mathbf{u} d\theta,
\end{aligned}
\end{equation}
which completes the proof of the first assertion.

Applying relation (\ref{eq21}) gives that
\begin{align*}
& \left|K_{\psi}\left(\mathbf{b},\mathbf{u},\theta;\mathbf{b}^{\prime},\mathbf{u}^{\prime},\theta^{\prime}\right)\right| \\
&=\Big|\Big\langle \mathbf{e}^{\mathbf{i}_{\mathbf{n}} (\mathbf{x}\cdot\mathbf{u}+\frac{A}{2B}\mathbf{b}^2-\frac{A}{2B}\mathbf{x}^2)}|\det A_\mathbf{u}|\psi\Big(R_{-\theta}A_{\mathbf{u}}(\mathbf{x}-\mathbf{b})\Big)C_{\psi}^{-1}, \\
&\mathbf{e}^{\mathbf{i}_{\mathbf{n}} (\mathbf{x}\cdot\mathbf{u'}+\frac{A}{2B}\mathbf{b}'^2-\frac{A}{2B}\mathbf{x}^2)}|\det A_\mathbf{u'}|\psi\Big(R_{-\theta^{\prime}}A_{\mathbf{u}^{\prime}}(\mathbf{x}-\mathbf{b}^{\prime})\Big)\Big\rangle_{L^{2}(\mathbb{R}^{n},C\ell_{n})} \Big|\\
&\leq\int_{\mathbb{R}^{n}}\Big||\det A_\mathbf{u}|\psi\Big(R_{-\theta}A_{\mathbf{u}}(\mathbf{x}-\mathbf{b})\Big)C_{\psi}^{-1}\Big| \Big||\det A_\mathbf{u'}|\psi\Big(R_{-\theta^{\prime}}A_{\mathbf{u}^{\prime}}(\mathbf{x}-\mathbf{b}^{\prime})\Big)\Big| d\mathbf{x} \\
&\leq\left\{\int_{\mathbb{R}^{n}}\left||\det A_\mathbf{u}|\psi\Big(R_{-\theta}A_{\mathbf{u}}(\mathbf{x}-\mathbf{b})\Big)C_{\psi}^{-1}\Big| d\mathbf{x}\right\}^{\frac{1}{2}}\left\{\int_{\mathbb{R}^{n}}\left||\det A_\mathbf{u'}|\psi\Big(R_{-\theta^{\prime}}A_{\mathbf{u}^{\prime}}(\mathbf{x}-\mathbf{b}^{\prime})\Big)\right| d\mathbf{x}\right\}^{\frac{1}{2}}\right\} \\
&=\left|\frac{1}{C_{\psi}}\right|^{1/2} \left\{\int_{\mathbb{R}^{n}}\left|\psi\left(R_{\theta}^{-1}\left(\mathbf{y}\right)\right)\right| \left|det A_{\mathbf{u}}\right|^{1-n}d\mathbf{y}\right\}^{1/2} \\
&\times\left\{\int_{\mathbb{R}^n}\left|\psi\left(R_{\theta'}^{-1}\left(\mathbf{y}'\right)\right)\right| \left|det A_{\mathbf{u}'}\right|^{1-n} d\mathbf{y}'\right\}^{1/2} \\
&=\left|\frac{|\det A_{\mathbf{u}}|^{1-n}|\det A_{\mathbf{u}^{\prime}}|^{1-n}}{C_{\psi}}\right|^{1/2}\left\{\|\psi\|_{L^{1}(\mathbb{R}^{n},C\ell_{n})}\right\}^{1/2}\left\{\|\psi\|_{L^{1}(\mathbb{R}^{n},C\ell_{n})}\right\}^{1/2} \\
&=\left|\frac{\left|\det A_{\mathbf{u}}\right|^{1-n}\left|\det A_{\mathbf{u}^{\prime}}\right|^{1-n}}{C_{\psi}}\right|^{1/2}\left\|\psi\right\|_{L^{1}(\mathbb{R}^{n},C\ell_{n})}.
\end{align*}
Thus, the proof of Theorem \ref{thm6} is concluded.
\end{proof}

\section{An example for the CLCST}
\label{sec4}

To better illustrate the workings of the proposed Clifford-valued linear canonical Stockwell transform, we present the following example. This example aims to provide a clear and practical demonstration of the transform's structure and properties, helping to elucidate its application within the framework of Clifford algebra and time-frequency analysis. Through this illustration, we hope to enhance understanding of the transform's effectiveness and theoretical underpinnings.

\begin{example}\label{exa1}
Consider the difference of Gaussian (DOG) functions, which is often used for edge detection or feature enhancement in signal processing. Mathematically, it can be written as:
\begin{equation}\label{eq35}
\psi(\mathbf{x})=\lambda^{-2}\exp\left\{-\frac{|\mathbf{x}|^2}{2\lambda^2}\right\}-\exp\left\{-\frac{|\mathbf{x}|^2}2\right\},\quad0<\lambda<1, \mathbf{x}\in\mathbb{R}^n.
\end{equation}
We first observe that
\begin{align*}
R_{-\theta}\Big(A_{\mathbf{u}}(\mathbf{x}-\mathbf{b})\Big)
&=\begin{pmatrix}\cos\theta&-\sin\theta&\ldots&0\\
\sin\theta&\cos\theta&\ldots&0\\
\vdots&\vdots&\ldots&\vdots\\
0&0&\ldots&1
\end{pmatrix}
\begin{pmatrix}u_1&0&\ldots&0\\
0&u_2&\ldots&0\\
\vdots&\vdots&\ldots&\vdots\\
0&0&\ldots&u_n
\end{pmatrix} 
\begin{pmatrix}
x_1-b_1\\
x_2-b_2\\
\vdots\\
x_n-b_n
\end{pmatrix}\\ 
&=\begin{pmatrix}
(u_1x_1-u_1b_1)\cos\theta-(u_2x_2-u_2b_2)\sin\theta\\
(u_1x_1-u_1b_1)\sin\theta+(u_2x_2-u_2b_2)\cos\theta\\
u_3x_3-u_3b_3\\
\vdots\\
u_nx_n-u_nb_n
\end{pmatrix}.
\end{align*}
Thus, the set of analysing functions $\psi_{M,\mathbf{b},\mathbf{u}}^\theta(\mathbf{x})$ is expressed as:
\begin{align*}
\psi_{M,\mathbf{b,u}}^{\theta}(\mathbf{x})& =\left|\det A_{\mathbf{u}}\right| \mathbf{e}^{\mathbf{i}_\mathbf{n}(\frac{A}{2B}\mathbf{b}^2+\mathbf{u}\cdot\mathbf{x}-\frac{A}{2B}\mathbf{x}^2)}
\psi\Big(R_{-\theta}\big(A_{\mathbf{u}}(\mathbf{x}-\mathbf{b})\big)\Big) \\
&=\left|\det A_{\mathbf{u}}\right| \mathbf{e}^{\mathbf{i}_\mathbf{n}(\frac{A}{2B}\mathbf{b}^2+\mathbf{u}\cdot\mathbf{x}-\frac{A}{2B}\mathbf{x}^2)}\left[\lambda^{-2}\exp\left\{\frac{-\left|R_{-\theta}(A_{\mathbf{u}}(\mathbf{x}-\mathbf{b}))\right|^{2}}{2\lambda^{2}}\right\}\right. \\
&- \exp\left\{\frac{-\big|R_{-\theta}\big(A_{\mathbf{u}}(\mathbf{x}-\mathbf{b})\big)\big|^2}{2}\right\} \\
&=\left|\det A_{\mathbf{u}}\right|\exp\left\{i_n[(x_1u_1+x_2u_2+\cdots+x_nu_n)+\frac{A}{2B}\mathbf{b}^2-\frac{A}{2B}(x_1^2+x_2^2+\cdots+x_n^2)]\right\} \\
&\times\left[\lambda^{-2} \exp\left\{ -\frac{1}{2\lambda^2}\Big[\big((u_1x_1-u_1b_1)\cos\theta-(u_2x_2-u_2b_2)\sin\theta\big)^2\right.\right] \\
&+\left((u_1x_1-u_1b_1)\sin\theta+(u_2x_2-u_2b_2)\cos\theta\right)^2 \\
&\left.+(u_3x_3-u_3b_3)^2+\cdots+(u_nx_n-u_nb_n)^2\right] \\
&- \exp\Bigg\{ - \frac{1}{2}\Big[\big((u_1x_1-u_1b_1)\cos\theta-(u_2x_2-u_2b_2)\sin\theta\big)^2 \\
&+\left((u_1x_1-u_1b_1)\sin\theta+(u_2x_2-u_2b_2)\cos\theta\right)^2 \\
&+(u_{3}x_{3}-u_{3}b_{3})^{2}+\cdots+(u_{n}x_{n}-u_{n}b_{n})^{2}\Big]\Biggr\}\Biggr].
\end{align*}
Consequently, the Clifford-valued linear canonical Stockwell transform, using the window function, can be computed as:
\begin{align*}
&\left(CLC\mathbb{S}^M_{\psi}f\right)(\mathbf{b},\mathbf{u},\theta) \\
&=\frac{\left|\det A_{\mathbf{u}}\right|}{\lambda^2(2\pi)^{n/2}}\int_{\mathbb{R}^n}f(\mathbf{x})\exp\left\{\frac{-1}{2\lambda^2}\Big[\left((u_1x_1-u_1b_1)\cos\theta-(u_2x_2-u_2b_2)\sin\theta\right)^2\right. \\
&+\Big((u_1x_1-u_1b_1)\sin\theta+(u_2x_2-u_2b_2)\cos\theta\Big)^2+(u_3x_3-u_3b_3)^2 \\
&\left.+\cdots+\left(u_nx_n-u_nb_n\right)^2\right]\Bigg\}\times\exp\left\{-i_n[(x_1u_1+x_2u_2+\cdots+x_nu_n)+\frac{A}{2B}\mathbf{b}^2-\frac{A}{2B}(x_1^2+x_2^2+\cdots+x_n^2)]\right\}d\mathbf{x} \\
&-\frac{|\det A_{\mathbf{u}}|}{(2\pi)^{n/2}}\int_{\mathbb{R}^n}f(\mathbf{x}) \exp\left\{-\frac{1}{2}\Big[\big((u_1x_1-u_1b_1)\cos\theta-(u_2x_2-u_2b_2)\sin\theta\big)^2\right] \\
&+\left((u_1x_1-u_1b_1)\sin\theta+(u_2x_2-u_2b_2)\cos\theta\right)^2+\left(u_3x_3-u_3b_3\right)^2\\
&\left.+\cdots+\left(u_nx_n-u_nb_n\right)^2\right]\Bigg\}\times\exp\left\{-i_n[(x_1u_1+x_2u_2+\cdots+x_nu_n)+\frac{A}{2B}\mathbf{b}^2-\frac{A}{2B}(x_1^2+x_2^2+\cdots+x_n^2)]\right\}d\mathbf{x}.
\end{align*}
For simplicity, we will compute the two-dimensional Clifford-valued Stockwell transform of the given function $f$ with respect to the difference of Gaussian function.
\begin{equation}\label{eq36}
\psi(x_1,x_2)=\lambda^{-2}\exp\left\{-\frac{|x_1|^2+|x_2|^2}{2\lambda^2}\right\} - \exp\left\{-\frac{|x_1|^2+|x_2|^2}{2}\right\}
\end{equation}
at different orientations, as follows:
\begin{align*}
\left(CLC\mathbb{S}^M_\psi f\right)(\mathbf{b},\mathbf{u},\pi/2)
&=\frac{|u_1u_2|}{2\pi\lambda^2}\exp\left\{-\frac{(u_1b_1)^2+(u_2b_2)^2}{2\lambda^2}\right\} \\
&\times\int_{\mathbb{R}^2}f(x_1,x_2)\exp\left\{-\frac{(u_1x_1)^2+(u_2x_2)^2}{2\lambda^2}\right\}\\
&\times\exp\left\{-i_2[(x_1u_1+x_2u_2)+\frac{A}{2B}\mathbf{b}^2-\frac{A}{2B}(x_1^2+x_2^2)]\right\}dx_1dx_2 \\
&-\frac{|u_1u_2|}{2\pi}\exp\left\{-\frac{(u_1b_1)^2+(u_2b_2)^2}{2}\right\} \\
&\times\int_{\mathbb{R}^2}f(x_1,x_2)\exp\left\{-\frac{(u_1x_1)^2+(u_2x_2)^2}{2}\right\}\\
&\times\exp\left\{-i_2[(x_1u_1+x_2u_2)+\frac{A}{2B}\mathbf{b}^2-\frac{A}{2B}(x_1^2+x_2^2)]\right\}dx_1dx_2.
\end{align*}
For the choice of parameters $\theta=\pi/2$, $(b_1,b_2) = (0,0)$, and $\lambda = 1/2$, the Clifford-valued linear canonical Stockwell transform of the function $f(x_1,x_2)=e^{-(x_1^2+x_2^2)}$ can be computed as
\begin{align*}
&\left(CLC\mathbb{S}^M_{\psi}f\right)(\mathbf{0},\mathbf{u},\pi/2) \\
&=\frac{2\left|u_{1}u_{2}\right|}{\pi}\int_{\mathbb{R}^{2}}\exp\left\{-(1+2u_{1}^{2})x_{1}^{2}-(1+2u_{2}^{2})x_{2}^{2}-i_{2}[(u_{1}x_{1}+u_{2}x_{2})-\frac{A}{2B}(x_1^2+x_2^2)]\right\}dx_{1}dx_{2} \\
&-\frac{|u_{1}u_{2}|}{2\pi}\int_{\mathbb{R}^{2}}\exp\left\{-\frac{(1+2u_{1}^{2})x_{1}^{2}+(1+2u_{2}^{2})x_{2}^{2}}{2}-i_{2}[(u_{1}x_{1}+u_{2}x_{2})-\frac{A}{2B}(x_1^2+x_2^2)]\right\} dx_{1}dx_{2} \\
&=\frac{2\left|u_1u_2\right|}{\pi}\int_{\mathbb{R}}\exp\left\{-(1+2u_1^2)x_1^2-i_2u_1x_1+i_2\frac{A}{2B}x_1^2\right\} dx_1 \\
&\times\int_{\mathbb{R}}\exp\left\{-(1+2u_2^2)x_2^2-i_2u_2x_2+i_2\frac{A}{2B}x_2^2\right\} dx_2 \\
&-\frac{|u_1u_2|}{2\pi} \int_{\mathbb{R}}\exp\left\{-\frac{(1+2u_1^2)x_1^2}{2}-i_2u_1x_1+i_2\frac{A}{2B}x_1^2\right\} dx_1 \\
&\times\int_{\mathbb{R}}\exp\left\{-\frac{(1+2u_1^2)x_2^2}{2}-i_2u_2x_2+i_2\frac{A}{2B}x_2^2\right\} dx_2
\end{align*}

\end{example}

\section{Conclusion}
\label{sec5}

This paper introduces the concept of the Clifford-valued linear canonical Stockwell transform. This transform allows for the representation of Clifford-valued functions at various scales and positions, while preserving directional information. We also delve into the fundamental properties of this transform and suggest future research within the $L^{p}$ theoretical framework. Additionally, this study highlights potential applications of the Clifford-valued linear canonical Stockwell transform, particularly in compressed sensing of monogenic signals and large-scale data problems. Ultimately, this work lays the groundwork for extending the utility of these transforms in advanced signal processing tasks.

\section*{Acknowledgments}
This work was supported by the National Natural Science Foundation of China [No. 62171041].

\section*{Declarations}
The authors declare that they have no known competing financial interests or personal relationships that could have appeared to influence the work reported in this paper.

\section*{Data Availability Statement}
Not applicable.

\section*{Code Availability Statement}
Not applicable.

\bibliography{sn-bibliography}% common bib file

\end{document}